\documentclass[10pt]{amsart}
\usepackage{indentfirst}
\usepackage{amssymb}
\usepackage{amsmath}
\usepackage{graphics}
\usepackage{epsfig}
\usepackage[usenames]{color}
\usepackage{MnSymbol,wasysym}
\usepackage[utf8]{inputenc}
\usepackage{latexsym}
\usepackage{graphicx}
\usepackage[normalem]{ulem}
\usepackage{hyperref}
\usepackage[T1]{fontenc}
\usepackage{kantlipsum}
\usepackage[usenames,dvipsnames]{xcolor}
\usepackage[breakable, theorems, skins]{tcolorbox}
\usepackage{ mathrsfs }
\usepackage[toc]{appendix}
\tcbset{enhanced}

\definecolor{light-gray}{gray}{0.95}

\catcode`@=11 \@addtoreset{equation}{section}
\renewcommand\theequation{\thesection.\@arabic\c@equation}
\catcode`@=12

\newcommand{\RR}{\mathbb{R}}
\newcommand{\PP}{\mathbb{P}}

\expandafter\chardef\csname pre amssym.def
at\endcsname=\the\catcode`\@ \catcode`\@=11
\def\undefine#1{\let#1\undefined}
\def\newsymbol#1#2#3#4#5{\let\next@\relax
 \ifnum#2=\@ne\let\next@\msafam@\else
 \ifnum#2=\tw@\let\next@\msbfam@\fi\fi
 \mathchardef#1="#3\next@#4#5}
\def\mathhexbox@#1#2#3{\relax
 \ifmmode\mathpalette{}{\m@th\mathchar"#1#2#3}%
 \else\leavevmode\hbox{$\m@th\mathchar"#1#2#3$}\fi}
\def\hexnumber@#1{\ifcase#1 0\or 1\or 2\or 3\or 4\or 5\or 6\or 7\or 8\or
 9\or A\or B\or C\or D\or E\or F\fi}

\font\teneufm=eufm10 \font\seveneufm=eufm7 \font\fiveeufm=eufm5
\newfam\eufmfam
\textfont\eufmfam=\teneufm \scriptfont\eufmfam=\seveneufm
\scriptscriptfont\eufmfam=\fiveeufm

\catcode`\@=\csname pre amssym.def at\endcsname

\definecolor{light-gray}{gray}{0.95}

\newcommand{\eqn}{\begin{eqnarray}}
\newcommand{\een}{\end{eqnarray}}

\newtheorem {Theorem}  {Theorem}

\numberwithin{Theorem}{section}

\newtheorem{Proposition}[Theorem]{Proposition}

\hypersetup{
  colorlinks   = true,    % Colours links instead of ugly boxes
  urlcolor     = blue,    % Colour for external hyperlinks
  linkcolor    = blue,    % Colour of internal links
  citecolor    = red      % Colour of citations
}

\newcommand{\ee}{\varepsilon}

\renewcommand{\a}{\alpha}

\renewcommand{\div}{\mbox{div}}

\begin{document}

\title[Asymptotic Behaviour of Oldroyd-B Fluids]{The Asymptotic Behaviour of Oldroyd-B Fluids is Almost Newtonian}

\author[M.  Hieber]{Matthias Hieber}
\address[M. Hieber]{Fachbereich Mathematik, TU Darmstadt, Schlossgartenstrasse 7, D-64289 Darmstadt, Germay} 
\email{hieber@mathematik.tu-darmstadt.de}

\author[T.H. Nguyen]{Thieu Huy Nguyen}
\address[T.H. Nguyen]{
Faculty of Mathematics and Informatics,
 Hanoi University of Science and Technology,
 Khoa Toan - Tin, Dai hoc Bach khoa Hanoi,
 1 Dai Co Viet,  Bach Mai, Hanoi, Vietnam}
\email{huy.nguyenthieu@hust.edu.vn}

\author[C. J. Niche]{C\'esar J. Niche}
\address[C.J. Niche]{Departamento de Matem\'atica Aplicada, Instituto de Matem\'atica. Universidade Federal do Rio de Janeiro, CEP 21941-909, Rio de Janeiro - RJ, Brazil}
\email{cniche@im.ufrj.br}

\author[C. F.  Perusato]{Cilon F.  Perusato}
\address[C.F.  Perusato]{Institut Camille Jordan, Université Lyon 1, 69622, Villeurbanne, France and Department of Mathematics,  Universidade Federal de Pernambuco,  Recife, PE 50740-560, Brazil}
\email{cilon.perusato@ufpe.br}



\begin{abstract} Consider a viscoelastic fluid of  Oldroyd-B type. It is  shown that its  stress tensor $\tau$ and its Newtonian deformation tensor $D(u)$ decay at the same rate,
 while the elastic part $\ee=\tau-2\omega D(u)$ decays faster. As a consequence, the stress tensor of a viscoelastic fluid exhibits an almost Newtonian behaviour for large times. 
   \end{abstract}

\maketitle

\section{Introduction}
\label{introduction}
The study of estimates for decay and of long time behaviour of solutions to equations arising in Fluid Mechanics has a long history. For example, there are many such results for the Navier-Stokes equations in $\mathbb{R} ^n$, starting with the work by Kato \cite{MR760047}, Masuda \cite{MR767409}, M.E. Schonbek \cite{MR775190, MR837929} and Wiegner \cite{MR881519}. In this article we consider decay estimates for solutions to the equations for  incompressible Oldroyd-B fluids  in $\RR^3$ describing  certain classes of viscoelastic fluids.
Given $\omega \in (0,1)$ and $a\in [-1,1]$, these read as 
\begin{equation}
\left\{
\begin{aligned}
Re \left(\partial_t u + (u \cdot \nabla) u \right) - (1 - \omega) \Delta u + \nabla P & = \div\,  \tau,  & \mbox{ in } \RR^3 \times (0,\infty)   \\
\div \,  u & = 0, & \mbox{ in } \RR^3 \times (0,\infty)  \\
We \left(\partial_t \tau + u \cdot \nabla \tau + g_a (\tau, \nabla u)   \right) + \tau & = 2 \omega \, D (u), & \mbox{ in } \RR^3 \times (0,\infty) \\
(u,  \tau) (0) & = \left( u_0, \tau_0\right) & \mbox{ in } \RR^3,   \\
\end{aligned}
\right. 
\label{eqn:oldroyd-b}
\end{equation}
where  $g_a (\tau, \nabla u) = \tau W(u) - W(u) \tau - a \left(D(u) \tau + \tau D(u) \right)$ and where $D(u) = \frac{1}{2} ( \nabla u + ( \nabla u)^T)$ denotes the deformation tensor and $W(u) = \frac{1}{2} ( \nabla u - ( \nabla u)^T)$ the vorticity tensor.  Here $u$ denotes the velocity of the fluid, $p$ its pressure, $\tau$ the symmetric constraint tensor and $Re$ and $We$ the associated  Reynolds and Weissenberg numbers, respectively. The nonlinear constitutive law for $\tau$ was proposed by Oldroyd \cite{MR94085} to describe certain classes of viscoelastic fluids.

The analysis of the above set of equations started with the pioneering paper by Guillop\'e and Saut \cite{MR1077577} in 1990, who proved local existence of strong solutions in suitable Sobolev spaces on bounded domains. This solution is global provided the data as well as the coupling constant $\omega$ are small enough.
The existence of a global weak solution in the case of $\RR^3$ was established by Lions and Masmoudi in \cite{MR1763488} for $a=0$. 
For further results in these  directions we refer to the works by Constantin and Kliegl \cite{MR2989441}, Fang and Zhi \cite{MR3473592},  Kupferman, Mangoubi and Titi \cite{MR2398006},
and Lei, Masmoudi and Zhou \cite{MR2558169}. Similar results, but for exterior domains, can be found in  Fang, Hieber and Zhi \cite{MR3096521}.
Properties of models in which either some dissipation or damping in \eqref{eqn:oldroyd-b} is not present have been studied by Constantin, Wu, Zhao and Zhi \cite{MR4350254},
Elgindi and Liu \cite{MR3349425}, Elgindi and Rousset \cite{MR3403757}, Wang, Wu, Xu and Zhong \cite{MR4348795} and Zhu \cite{MR3762094}. Finally, for a survey article also  
describing some physical aspects of the Oldroyd-B model, see Renardy and Thomases \cite{MR4266315}.

The main goal of this article is to provide  precise decay estimates for solutions $z = (u, \tau)$ to the Oldroyd-B model using the decay character of the initial data $u_0$ and $\tau_0$. Let us emphasize that our approach does not only allows us to improve known decay estimates for viscoelastic fluids, but also to prove the surprising fact that the asymptotic behaviour of
Oldroyd-B fluids is almost Newtonian. In fact,
we show that the full stress tensor and its Newtonian part decay at the same rate, while the viscoelastic part decays at a strictly faster rate. As a consequence, the full stress tensor and the
Newtonian deformation tensor become asymptotically aligned. 

We now describe the results mentioned above in a more precise way. In order to do so, we recall that to an initial datum $v_0 \in L^2 (\mathbb{R} ^n)$ we can associate a decay character $r^{\ast} (v_0)$ which is key for establishing sharp decay estimates for dissipative linear equations, see Section \ref{preliminaries} for details. In Theorem \ref{thm-a-equal-zero} we consider
the corotational case in \eqref{eqn:oldroyd-b}, i.e.  $a=0$, and prove decay estimates for weak solutions with arbitrary large initial data, of the form
$$
\Vert z(t) \Vert _{L^2}^2 \leq C     (1 + t) ^{ -\min \left\{ \frac{3}{2}, \frac{3}{2} +  \min \{  r^{\ast} \left(u_0 \right),  1 + r^{\ast} \left(\tau_0 \right)\, \}\, \right\}}, \quad  t>0,
$$
hereby improving previous results by Hieber, Wen and Zi \cite{MR3909068}  and Chen, Li, Yao and Yao \cite{MR4643473}.

Furthermore, our approach also allows us to improve known decay estimates for strong solutions for $a \ne 0$, albeit for small data. In Theorem \ref{general-a} we prove that for
any given $0 < \omega < 1$, there is $\epsilon = \epsilon (\omega)$  such that if $\Vert \left( u_0, \tau_0 \right) \Vert _{H^2} \leq \epsilon $
then for $k = 0, 1, 2$ and large times we have 
$$ 
\Vert \nabla^k u(t) \Vert _{L^2}^2 +  \Vert \nabla^k \tau(t) \Vert _{L^2}^2    \leq C    (1 + t) ^{- \left( k + \min \left\{ \frac{3}{2}, \frac{3}{2} +  \min \{  r^{\ast} \left(u_0 \right),  1 + r^{\ast} \left(\tau_0 \right) \} \right\} \right)},
$$
and for $k = 0,1$ 

$$  \Vert \nabla^k \tau(t) \Vert _{L^2}^2    \leq C    (1 + t) ^{- \left( k + 1 + \min \left\{ \frac{3}{2}, \frac{3}{2} +  \min \{  r^{\ast} \left(u_0 \right),  1 + r^{\ast} \left(\tau_0 \right) \} \right\} \right)}. 
$$
These results improve bounds obtained by  Hieber, Wen and Zi \cite{MR3909068}  and Chen, Li, Yao and Yao \cite{MR4643473}.

These estimates for $\tau$ and $Du$ lead us to the question of whether the elastic part $\ee$ of the stress tensor, i.e. $\ee = \tau  - 2 \omega D(u)$, allows for faster decay estimates than those. This is indeed the case. In fact, under the same assumptions as above, in Theorem \ref{thm:error} we prove that  
\begin{equation*}
  \Vert  \varepsilon(t) \Vert _{L^2}^2    \leq C     (1 + t) ^{- \left( 2 + \min \left\{\frac{3}{2}, \frac{3}{2} +  \min \{\,  r^{\ast} \left(u_0 \right),  1 + r^{\ast} \left(\tau_0 \right)\, \}\, \right\} \right)},
  \quad\,  t>0. 
\end{equation*}
This suggests that, as $\varepsilon$ decays faster than $\tau$ and $D(u)$, the latter would   tend to align in such a way  that the viscoelastic stress tensor asymptotically
matches the Newtonian deformation tensor. Note, however, that the above estimates  are just  upper bounds and do not preclude the possibility that these quantities decay faster than as described above. 

In Theorem \ref{thm:lower} we are able to prove that for certain classes of initial data we have lower bounds for $\tau$ and $D(u)$ that are equivalent to  their upper bounds. For example, for $( u_0,\tau_0 )$ satisfying $r^{\ast} (u_0) \leq 1 + r^{\ast} (\tau _0)$ and $r^{\ast} (u_0) \leq 0$, we obtain

\begin{displaymath}
  C_1    (1 + t) ^{- \left( \frac{5}{2}  + r^{\ast} \left(u_0 \right) \right)} \leq \Vert D(u) (t) \Vert _{L^2} ^2, \Vert \tau (t) \Vert _{L^2} ^2 \leq C_2    (1 + t) ^{- \left( \frac{5}{2}  + r^{\ast} \left(u_0
      \right) \right)}, \quad   t>1.
\end{displaymath}
The conditions on the decay characters of initial data imply that

\begin{equation*}
\Vert  \varepsilon(t) \Vert _{L^2}^2    \leq C     (1 + t) ^{- \left(\frac{7}{2} + r^{\ast} (u_0) \right)}, \quad\mbox{ for all } \, t>0. 
\end{equation*}
Hence, the elastic part $\varepsilon$ decays faster than both the full  stress tensor  $\tau$ and its Newtonian part $Du$. A similar result is valid under the conditions $1 + r^{\ast} (\tau _0) \leq r^{\ast} (u_0)$ and $1 + r^{\ast} (\tau _0) \leq 0$. This is all summarized in Theorem \ref{thm:main-theorem}. Thus, we have shown that the faster decay of the elastic part of the stress tensor implies that the viscoelastic fluid exhibits an almost Newtonian behavior for large times.

The above estimates for strong solutions  are similar to results by Guterres, Melo, Niche, Perusato and Zingano \cite{MR4844564} in the context of micropolar fluids. It is shown there that the vorticity and micro-rotation of a $3D$ micropolar flow tend to align  asymptotically. For related results, see also Shang \cite{MR4801878}.

A different kind of limit behaviour for viscoleastic fluids  has been obtained by Bresch and Prange \cite{MR3176320} and  Molinet and Talhouk \cite{MR2377290}. They
proved, assuming that the Weissenberg number $We$ goes to zero in \eqref{eqn:oldroyd-b}, that  solutions $u$ converge in appropriate spaces to solutions to the Navier-Stokes equations while  $\tau$
converges to $ 2 \omega D(u)$.

This article is organized as follows. In Section 2  we start by recalling  properties of the decay character of $v_0 \in L^2 (\mathbb{R}^n)$ and we then state our Theorems. In Section 3 we prove decay  results  for the linear part of \eqref{eqn:oldroyd-b}. In Section 4 we provide the proofs of Theorems \ref{thm-a-equal-zero} and \ref{general-a} concerning decay of
solutions and its derivatives and prove, in Theorems \ref{thm:error} and \ref{thm:lower}, the almost Newtonian behaviour of the Oldroyd-B flow.

\section{Preliminaries and Main Results}

\subsection{Preliminaries} \label{preliminaries} Let us start by recalling several  properties of the decay character of a function $v_0 \in L^2(\RR^n)$. This number describes the  ``algebraic order'' of $v_0$ near the origin,
comparing $ |\widehat{v_0} (\xi)|$ to $f(\xi) = |\xi|^{r}$ at $\xi = 0$, for some $r \in \RR$. This concept  was first introduced by  Bjorland and M.E. Schonbek in \cite{MR2493562} and refined later  by Niche and M.E. Schonbek \cite{MR3355116} and Brandolese \cite{MR3493117}.

For  $v_0 \in L^2(\RR^n)$, $r \in \left(- \frac{n}{2}, \infty\right)$ and a ball $B(\rho)$ at the origin with radius $\rho$, we consider

\begin{displaymath}
P_r(v_0):=\lim _{\rho \to 0} \rho ^{-2r-n} \int _{B(\rho)} \bigl |\widehat{v_0} (\xi) \bigr|^2 \, d \xi.    
\end{displaymath}
The decay character of $ v_0$, denoted by $r^{\ast} = r^{\ast}( v_0)$ is the unique  $r \in \left( -\frac{n}{2}, \infty \right)$
satisfying  $0 < P_r (v_0) < \infty$, provided this number exists. For a discussion of necessary and sufficient conditions for the existence of this limit, see Brandolese \cite{MR3493117} and Brandolese, Perusato and Zingano \cite{MR4700578}.

The decay character may be calculated for many functions. For example, given  $v_0 \in L^p (\RR^n) \cap L^2 _{\sigma} (\RR ^n)$ for $1 \leq p < 2$, we have that
$r^{\ast} (v_0) = - n ( 1 - \frac{1}{p})$, see e.g. \cite{MR3565380}.

For  initial data  $v_0 \in H^s (\RR^n)$, where $ s > 0$, there is a natural relation between the decay character of $\Lambda ^s \, v_0 := (- \Delta)^{\frac{s}{2}} v_0$ and that of $v_0$. Setting  $r^{\ast} _s (v_0) := r^{\ast} ( \Lambda^s v_0 )$, it follows from Niche and M.E. Schonbek \cite{MR3355116} that
$ \frac{n}{2} +s< r_s^{\ast}(v_0) < \infty $ and $r_s^{\ast}(v_0) = s + r^{\ast} (v_0)$, provided  $-\frac{n}{2} < r^{\ast} (v_0) < \infty$. 

Given a Hilbert space $X$ consider a linear pseudodifferential operator $\mathcal{L}: X^n \to \left( L^2 (\RR^n) \right) ^n$. Assume that $\mathcal{L}$ is diagonalizable, in the sense that
the associated symbol $ M(\xi)$ satisfies $M(\xi) = P^{-1} (\xi) D(\xi) P(\xi)$ for a.a. $\xi \in \RR^n$,  where $P(\xi) \in O(n)$ and $D$ satisfies $D(\xi) = - c_i |\xi|^{2\a} \delta _{ij}$ for $c_i > c>0$ and $0 < \a \leq 1$. An example of such operator is provided by the  fractional vectorial Laplacian.  If $v$ satisfies the equation $v_t = \mathcal{L} v$ for all $t>0$,  then 
$$
\frac{1}{2} \frac{d}{dt} \Vert v(t) \Vert _{L^2} ^2 \leq  - C  \int _{\RR^n} |\xi|^{2 \a} |\widehat{v}|^2 \, d \xi
$$
and the $L^2$-norm of $v$ decreases. Niche and  M.E. Schonbek  \cite{MR3355116} obtained  sharp upper and lower bounds for the norms of solutions to such systems. More precisely, let $v_0 \in L^2 (\RR^n)$ with  decay character $r^{\ast} (v_0) = r^{\ast}$ and  assume that $v$ is  a solution to  $v_t = \mathcal{L}v$ with initial data $v_0$. Then, if $- \frac{n}{2 } < r^{\ast}< \infty$, there exist constants $C_1, C_2> 0$ such that

\begin{equation}
\label{eqn:NS15}
C_1 (1 + t)^{- \frac{1}{\a} \left( \frac{n}{2} + r^{\ast} \right)} \leq \Vert v(t) \Vert _{L^2} ^2 \leq C_2 (1 + t)^{- \frac{1}{\a} \left( \frac{n}{2} + r^{\ast} \right)}. 
\end{equation}
Let us note here that the decay of solutions to the heat equation in $\mathbb{R} ^n$ is given by \eqref{eqn:NS15} with $\alpha = 1$. 

\subsection{Main results} Our first result concerns the decay of weak solutions in the case $a = 0$.  Note that existence of global weak solutions was proved in this case by Lions and Masmoudi \cite{MR1763488}, but their
uniqueness is not  known until today.

\begin{Theorem}\label{thm-a-equal-zero}
Let $a = 0$ and  $(u_0,  \tau_0) \in L^2_\sigma (\RR^3) \times L^2 (\RR^3)$, with $-\frac{3}{2} < r^{\ast} (u_0), r^{\ast}(\tau_0) < \infty $.  Then, for any weak solution
$z = (u, \tau)$ to \eqref{eqn:oldroyd-b}, there exists a constant $C>0$ such that
\begin{equation}\label{eqn:decay-theorem-a0}
\Vert z(t) \Vert _{L^2}^2 \leq C     (1 + t) ^{ -\min \left\{ \frac{3}{2}, \frac{3}{2} +  \min \{r^{\ast}(u_0),  1 + r^{\ast}(\tau_0) \} \right\}  },  \quad t>0.
\end{equation}
\end{Theorem}

Let us note that \eqref{eqn:decay-theorem-a0} improves some known decay estimates. In fact, if  $u_0 \in L^2_\sigma (\RR^3) \cap L^p (\RR^3)$ and $\tau_0 \in L^2 (\RR^3) \cap L^p (\RR^3)$,
it is shown for $p = 1$ in Hieber, Wen and Zi \cite{MR3909068},  and for $1 \leq p <2$ in Chen, Li, Yao and Yao \cite{MR4643473} , that 
\begin{equation*}
    \Vert  z(t) \Vert _{L^2} ^2 \leq C (1+t) ^{-\frac{3}{2}\left(\frac{2}{p} - 1\right)}.
\end{equation*}
From Section \ref{preliminaries} we know that $r^{\ast} (f) = - 3 ( 1 - \frac{1}{p})$ for $f \in L^2 (\RR^3) \cap L^p (\RR^3)$ and $1 \leq p < 2$. Hence, our Theorem \ref{thm-a-equal-zero} covers
in particular the above estimate.

Our second result deals with the decay of strong solutions in the general case  $a \in [-1,1]$, for small enough initial data. For the existence of such global, strong solutions, we refer to Guillop\'e and Saut \cite{MR1077577}.

\begin{Theorem}\label{general-a}
Let $a \in [-1,1]$ and $(u_0,\tau_0) \in H^2 _{\sigma} (\RR^3) \times H^2 (\RR^3)$ satisfying  $- \frac{3}{2} < r^{\ast} (u_0), r^{\ast}(\tau_0) < \infty$. Then, for $k=0,1,2$ there exist
$\varepsilon = \varepsilon(\omega) >0$ and $T = T(\omega)$ such that if $\Vert ( u_0, \tau_0) \Vert _{H^2} \leq \varepsilon $, there exists a constant $C>0$ such that   
\begin{equation} \label{eqn:higher-derivatives} 
\Vert \nabla^k u(t) \Vert _{L^2}^2 +  \Vert \nabla^k \tau(t) \Vert _{L^2}^2    \leq C    (1 + t) ^{- \left( k + \min \left\{ \frac{3}{2}, \frac{3}{2} +  \min \{  r^{\ast} \left(u_0 \right),  1 + r^{\ast} \left(\tau_0 \right) \} \right\} \right)} ,\quad  t>T.
\end{equation}
Moreover, for $j=0,1$ there exists a constant $C>0$ such that
\begin{equation}\label{eqn:high-deriv-tau}
\Vert \nabla^j \tau(t) \Vert _{L^2}^2   %\frac{1}{2} 
%\Vert \tau (t) \Vert _{L^2} ^2   
\leq C  (1 + t) ^{-  \left( j+1 + \min \left\{ \frac{3}{2}, \frac{3}{2} +  \min \{  r^{\ast} \left(u_0 \right),  1 + r^{\ast} \left(\tau_0 \right) \} \right\} \right)}, \quad  t>T.
\end{equation}
\end{Theorem}

We note that Theorem \ref{general-a} generalizes  previous results in  Hieber, Wen and Zi \cite{MR3909068} and  Chen, Li, Yao and Yao \cite{MR4643473}, where for  $u_0 \in H^2_\sigma (\RR^3) \cap L^1 (\RR^3)$ and $\tau_0
\in H^2 (\RR^3) \cap L^1 (\RR^3)$ it was shown that for $k=0,1,2$ 

\begin{displaymath}
    \Vert \nabla^k u(t) \Vert _{L^2}^2 +  \Vert \nabla^k \tau(t) \Vert _{L^2}^2    \leq C    (1 + t) ^{- \left( k + \frac{3}{2} \right)}, \quad t>0. 
\end{displaymath}

We now address our main result, in which we consider the elastic part $\ee$ of the stress tensor given by the difference between $\tau$  and the Newtonian part $2 \omega D(u)$, i.e. we set 
$$
\ee = \tau  - 2 \omega D(u).
$$
Estimates \eqref{eqn:higher-derivatives} and \eqref{eqn:high-deriv-tau} in Theorem \ref{general-a} imply that both $\tau$ and $D(u)$ obey the same decay estimates. The next result says that there is even a  faster decay rate for $\ee$.  

\begin{Theorem}\label{thm:error}
Under the assumptions in Theorem \ref{general-a}, there is a constant $C>0$ such that    

\begin{equation}\label{eqn:epsilon-estimates}
  \Vert  \varepsilon(t) \Vert _{L^2}^2    \leq C     (1 + t) ^{- \left( 2 + \min \left\{\frac{3}{2}, \frac{3}{2} +  \min \{\,  r^{\ast} \left(u_0 \right),  1 + r^{\ast} \left(\tau_0
        \right)\, \}\, \right\} \right)}, \quad t>T.  
\end{equation}
\end{Theorem}
Estimate \eqref{eqn:epsilon-estimates} suggests that as $\varepsilon$ decays faster than $\tau$ and $D(u)$, these may  tend to align in such a way that the viscoelastic stress tensor
asymptotically  matches the Newtonian one. However, note that \eqref{eqn:higher-derivatives} and \eqref{eqn:high-deriv-tau} provide only  upper bounds, which do not preclude the possibility of these quantities having
a faster decay than \eqref{eqn:epsilon-estimates}.

We show next that for certain initial data $( u_0,  \tau_0 )$ it is possible to obtain upper as well as  lower bounds for $Du$ and $\tau$ with the same decay rate.

\begin{Theorem} \label{thm:lower} Let $r^{\ast} (u_0) \leq 1 + r^{\ast} (\tau _0)$ and $r^{\ast} (u_0) \leq 0$. Then there exists a constant $C>0$ such that 
\begin{equation}
\label{eqn:lower-bounds-for-final-estimate-first-case}
\Vert D(u) (t) \Vert _{L^2} ^2, \Vert \tau (t) \Vert _{L^2} ^2 \geq C    (1 + t) ^{- \left( \frac{5}{2}  + r^{\ast} \left(u_0 \right) \right)}, \quad t>1.
\end{equation}
On the other hand, if $1 + r^{\ast} (\tau _0) \leq r^{\ast} (u_0)$ and $1 + r^{\ast} (\tau _0) \leq 0$, then
\begin{equation}
\label{eqn:lower-bounds-for-final-estimate-second-case}
\Vert D(u) (t) \Vert _{L^2} ^2, \Vert \tau (t) \Vert _{L^2} ^2 \geq C    (1 + t) ^{- \left( \frac{7}{2}  + r^{\ast} \left(\tau_0 \right) \right)}, \quad \, t>1.
\end{equation}
    
\end{Theorem}
Similar lower bounds for a restricted class of initial data were obtained by Huang, Wang, Wen and Zi \cite{MR4335131}, \cite{corrigendum}.  Chen, Li, Yao and Yao \cite{MR4643473} claim to obtain lower bounds for the decay provided that  $u_0 \in L^1 (\RR^3)$ has nonzero mean. Note, however, that this condition is not  fulfilled for any  incompressible $u_0$, see page 721 of the article by  M.E. Schonbek, T. Schonbek and S\"uli \cite{MR1380452}.  Hence, their claim is not justified.

In the following Theorem we just combine the estimates proved in Theorems \ref{general-a}, \ref{thm:error} and \ref{thm:lower} in order to  state in a clearer way our main result: the faster decay of the elastic part $\varepsilon$
of the stress tensor implies that the viscolelastic stress tensor  exhibits an almost Newtonian behavior for large times.    

\begin{Theorem} \label{thm:main-theorem} Under the assumptions of Theorem \ref{general-a} we have:
    \begin{itemize}
        \item[a)] Assume that $r^{\ast} (u_0) \leq 1 + r^{\ast} (\tau _0)$ and $r^{\ast} (u_0) \leq 0$. Then there exist constants $C_1,C_2>0$ such that  
\begin{displaymath}
  C_1    (1 + t) ^{- \left( \frac{5}{2}  + r^{\ast} \left(u_0 \right) \right)} \leq \Vert D(u) (t) \Vert _{L^2} ^2, \Vert \tau (t) \Vert _{L^2} ^2
  \leq C_2    (1 + t) ^{- \left( \frac{5}{2}  + r^{\ast} \left(u_0 \right) \right)}, \quad  t> 1.
\end{displaymath}
Moreover, there exists a constant $C>0$ such that 
\begin{displaymath}
    \Vert  \varepsilon(t) \Vert _{L^2}^2    \leq C     (1 + t) ^{- \left( \frac{7}{2}  + r^{\ast} \left( u_0 \right) \right)}, \quad t> \max \left\{T,1\right\},
\end{displaymath}
so \begin{displaymath}
  \tau = 2 \omega\, D(u) + O \left( (1+t)^{- \left(  \frac{7}{2} + r^ {\ast} (u_0) \right)} \right). 
\end{displaymath}

        \item[b)] Assume that $1 + r^{\ast} (\tau _0) \leq r^{\ast} (u_0)$ and $1 + r^{\ast} (\tau _0) \leq 0$. Then there exist constants $C_1,C_2>0$ such that
\begin{displaymath}
  C_1     (1 + t) ^{- \left( \frac{7}{2}  + r^{\ast} \left(\tau_0 \right) \right)} \leq \Vert D(u) (t) \Vert _{L^2} ^2, \Vert \tau (t) \Vert _{L^2} ^2
  \leq C_2      (1 + t) ^{- \left( \frac{7}{2}  + r^{\ast} \left(\tau_0 \right) \right)}, \quad\forall \, t>1.
\end{displaymath}
Moreover, there exists a constant $C>0$ such that 
\begin{displaymath}
    \Vert  \varepsilon(t) \Vert _{L^2}^2    \leq C     (1 + t) ^{- \left( \frac{9}{2}  + r^{\ast} \left( \tau_0 \right) \right)}, \quad t>\max \left\{T,1\right\},
\end{displaymath}
so \begin{displaymath}
  \tau = 2 \omega\, D(u) + O \left( (1+t)^{- \left(  \frac{9}{2} + r^ {\ast} (\tau_0) \right)} \right). 
\end{displaymath}        
    \end{itemize}
\end{Theorem}

\section{Analysis of the Linear Part}
In this section we prove decay estimates for the linear part of \eqref{eqn:oldroyd-b}, which is given by the equations 
\begin{align}
\label{eqn:linear-part-oldroyd}
\partial_t u - (1 - \omega) \Delta u - \PP \, div \, \tau & = 0,  \notag \\
\partial_t \tau + \tau - 2 \omega \, D (u) & = 0.
\end{align}
We first provide an explicit expression  in frequency space  for the solution $(u_l,\tau_L)^T$ of this equation given by    
$(u_L(t),\tau_L (t))^T = \mathcal{G} \ast (u_0,\tau_0)^T $, 
where $\mathcal{G}$ denotes the fundamental solution to \eqref{eqn:linear-part-oldroyd} and  $u_0, \tau_0 \in L^2 (\RR^d)$.

\begin{Proposition} (Zi, Fang and Zhang \cite{MR3211863}, Lemma 3.1). \label{zi-fang-zhang}
Let $(u_L, \tau _L)$ be a solution to \eqref{eqn:linear-part-oldroyd} and for $t >0, \, \xi \in \RR^d$ set
\begin{align*}
  \mathcal{A} (\xi, t):= \frac{e^{\lambda_+ t} - e^{\lambda_- t}}{\lambda_+ - \lambda _-},\quad \mathcal {B} (\xi,t):= \frac{ \left( \lambda_+ + 1 \right) e^{\lambda_+ t} -
  \left( \lambda_- + 1 \right)  e^{\lambda_- t}}{\lambda_+ - \lambda _-}, \quad   \mathcal {C} (\xi,t) := \frac{ \left( \lambda_- + 1 \right) e^{\lambda_+ t} -
  \left( \lambda_+ + 1 \right)  e^{\lambda_- t}}{\lambda_+ - \lambda _-},
\end{align*}
where
\begin{displaymath}
\lambda _{\pm} (\xi) = \frac{- (1 + (1 - \omega)|\xi|^2) \pm \sqrt{\left(1 + (1 - \omega)|\xi|^2) \right)^2 - 4|\xi|^2}}{2}.
\end{displaymath}
Then,  in frequency space the solution $\widehat{L} (\xi, t) = \left( \widehat{u_L} (\xi, t),  \widehat{\tau _L} (\xi, t) \right)$ to  \eqref{eqn:linear-part-oldroyd}  is given by
\begin{displaymath}
  \left( \widehat{u_L} \right) _j  (\xi,t)= \mathcal {B} (\xi,t) \left( \widehat{u_0} \right)_j (\xi) + \mathcal{A} (\xi, t) \, i \, \xi_l
  (\delta _{jk} - \frac{\xi _j \xi_k}{|\xi|^2} ) \left(\widehat{\tau _0} \right)_{lk} (\xi), 
\end{displaymath}
and
\begin{align}
  \left( \widehat{\tau_L} \right) _{jk} (\xi,t) & = e^{-t} \left( \widehat{\tau_0}\right) _{jk} (\xi)  \notag \\   &  - \left( \mathcal {C} (\xi,t)   + e^{-t} \right) \left( \frac{\xi _j \xi_l}{|\xi|^2} \left( \delta_{km} - \frac{\xi _k \xi_m}{|\xi|^2} \right)
   \left( \widehat{\tau_0}\right) _{lm} (\xi)  + \frac{\xi _k \xi_l}{|\xi|^2} \left(  \delta_{jm} - \frac{\xi _j \xi_m}{|\xi|^2} \right) \left( \widehat{\tau_0}\right) _{lm} (\xi) \right) \notag \\ & + \, i \, \omega \mathcal{A} (\xi, t) \left(\xi_j \left( \widehat{u_0} \right) _k (\xi) + \xi_k
     \left( \widehat{u_0} \right) _j (\xi)  \right). \notag
\end{align}
\end{Proposition} 

The following pointwise estimates for $\mathcal{A},  \mathcal{B}$ and $\mathcal{C}$ were obtained by Hieber, Wen and Zi in \cite{MR3909068}.

\begin{Proposition} (Hieber, Wen and Zi \cite{MR3909068}, Lemma 2.2) \label{hieber-wen-zi}
Given $R > 0$, there exist positive constants

\begin{align*}
\theta (\omega, R) & = \frac{1}{2} \min \left\{ \frac{1 - \omega}{2}, \frac{1}{1 + R^2(1 - \omega)} \right\} \\
C_1 (\omega, R) & = \max \left\{ \frac{4 \sqrt{1  -\omega}}{\min \left\{\sqrt{\omega}, 1 - \sqrt{\omega} \right\}},  8 (1 + \sqrt{\omega})^2 \max \left\{\frac{2}{1 - \omega}, 1 + R^2(1 - \omega)  \right\} \right\} \\
C_2 (\omega, R) & = 2 \max \left\{1 + (1 - \omega) R^2, \sqrt{\omega} R \right\} C_1 (\omega, R)\\
C_3 (\omega, R) & = \max \left\{C_2 (\omega, R), \frac{\max \left\{3 \omega \sqrt{1 - \omega}, \frac{2(1 - \sqrt{\omega})}{\sqrt{1 - \omega}} \right\}}{\min \left\{\sqrt{\omega}, 1 - \sqrt{\omega} \right\}} \right\}
\end{align*}
such that for $|\xi| \leq R$ and $t>0$ we have

\begin{align*}
|\mathcal{A} (\xi, t)| & \leq C_1(\omega, R) e ^{- \theta (\omega, R) |\xi|^2 t}, \\
|\mathcal{B} (\xi,t)| & \leq C_2(\omega, R) e ^{- \theta (\omega, R) |\xi|^2 t}, \\
|\mathcal{C} (\xi,t)| & \leq C_3(\omega, R) \left( e ^{- \frac{\theta (\omega, R)}{4(1 + \sqrt{\omega})^2} t} + |\xi|^2 e ^{- \theta (\omega, R) |\xi|^2 t} \right). 
\end{align*}
\end{Proposition}

The decay estimates for the solution to the linear equation reads as follows.

\begin{Proposition}\label{decay-linear-part}
  Let $u_0, \tau_0 \in L^2 (\RR^d)$ with decay characters $-\frac{3}{2} < r^{\ast} \left(u_0 \right),  r^{\ast} \left(\tau_0 \right) < \infty $. Then there exists a constant $C>0$ such that
 the solution $(u_L ,\tau _L)$ to \eqref{eqn:linear-part-oldroyd} satisfies 

\begin{displaymath}
\omega \Vert u_L(t) \Vert _{L^2}^2  + \frac{1}{2} \Vert \tau_L (t) \Vert _{L^2} ^2 \leq C   (1 + t) ^{- \left( \frac{d}{2} + \min \{  r^{\ast} \left(u_0 \right),  1 +  r^{\ast} \left(\tau_0 \right) \} \right) }, \quad t>0. 
\end{displaymath}

\end{Proposition}

\begin{proof}
Multiplying  \eqref{eqn:linear-part-oldroyd} by $\left(2 \omega \, u_L , \tau_L \right)$,  integrating  by parts and adding  leads to   
\begin{equation}
\label{eqn:energy-equality-decay-linear}
\frac{d}{dt} \left( \omega \Vert u_L(t) \Vert _{L^2}^2 + \frac{1}{2} \Vert \tau_L (t) \Vert _{L^2} ^2 \right) + \Vert \tau_L (t) \Vert _{L^2} ^2 = - 2 \omega (1 - \omega) \Vert \nabla u_L (t) \Vert _{L^2} ^2.
\end{equation}
In frequency space consider the ball

\begin{displaymath}
B(t) = \left\{\xi \in \RR^d: |\xi| \leq \frac{g(t)}{\sqrt{1 - \omega}}   \right\},
\end{displaymath}
for some continuous, decreasing function $g$ with g(0) = 1, which will be  determined later. Then 

\begin{displaymath}
    - 2 \omega (1 - \omega) \Vert \nabla u_L (t) \Vert _{L^2} ^2   \leq - 2 \omega (1 - \omega) \int _{B(t) ^c} |\xi|^2 |\widehat{u_L} (\xi, t) | ^2 \, d \xi  \leq - 2 \omega g^2(t) \int _{B(t) ^c}  |\widehat{u_L} (\xi, t) | ^2 \, d \xi,
\end{displaymath}
which together with \eqref{eqn:energy-equality-decay-linear} leads to

\begin{equation}
\label{ineq-decay-linear}
  \frac{d}{dt} \left( \omega \Vert u_L(t) \Vert _{L^2}^2  + \frac{1}{2} \Vert \tau_L (t) \Vert _{L^2} ^2 \right)  +  \Vert \tau_L (t) \Vert _{L^2} ^2 + 2 \omega \, g^2(t) \Vert u_L(t) \Vert _{L^2} ^2  \leq  2 \omega \, g^2(t) \int _{B(t)}  |\widehat{u_L} (\xi, t) | ^2 \, d \xi.
\end{equation}
We now establish a pointwise estimate for $|\widehat{u} (\xi, t) |$ in $B(t)$.  Proposition \ref{zi-fang-zhang} yields

\begin{displaymath}
|\widehat{u_L} (x,t)|^2 \leq C \left(|\mathcal{B} (\xi,t)| ^2 |\widehat{u_0} (\xi)|^2 +   |\xi|^2 |\mathcal{A} (\xi,t)| ^2  |\widehat{\tau_0} (\xi)|^2 \right).
\end{displaymath}
Suppose now that $|\xi| \leq g(t) \leq R$,  for some fixed $R$.  Then,  after choosing $g^2(t) = C (1 + t) ^{-1}$ for some $C  > 0$,  from this  and Proposition \ref{hieber-wen-zi} we obtain 

\begin{align*}
\int _{B(t)} |\widehat{u_L} (x,t)|^2 \, d \xi & \leq C \left(  \int _{B(t)} e^{- C t |\xi|^2 } |\widehat{u_0} (\xi) |^2 \, d \xi + \int _{B(t)} |\xi|^2 e^{- C t |\xi|^2 } |\widehat{\tau_0} (\xi) |^2 \, d \xi \right) \\ & \leq C \Vert e^{t \Delta} u_0 \Vert _{L^2} ^2 + C (1 + t) ^{-1} \Vert e^{t \Delta} \tau_0 \Vert _{L^2} ^2 \\ & \leq C (1 + t) ^{- \left( \frac{d}{2} + r^{\ast} \left(u_0 \right) \right)} + C (1 + t) ^{- \left( 1+  \frac{d}{2} + r^{\ast} \left(\tau_0 \right) \right)},
\end{align*}
where we used the decay of solutions to the heat equation, i.e., estimates in \eqref{eqn:NS15} with $\alpha = 1$. Plugging this into \eqref{ineq-decay-linear}  yields

\begin{align}
\label{eqn:fourier-splitting-before-integrating-factor}
\frac{d}{dt} \left( \omega \Vert u_L(t) \Vert _{L^2}^2  + \frac{1}{2} \Vert \tau_L (t) \Vert _{L^2} ^2 \right) & + \Vert \tau_L (t) \Vert _{L^2} ^2 + C  \omega \, (1 + t) ^{-1}  \Vert u_L(t) \Vert _{L^2} ^2 \notag \\ & \leq  C \omega \, (1 + t) ^{-1} \left(  (1 + t) ^{- \left( \frac{d}{2} + r^{\ast} \left(u_0 \right) \right)} +  (1 + t) ^{- \left( 1+  \frac{d}{2} + r^{\ast} \left(\tau_0 \right) \right)} \right).
\end{align}
Given  $A > 1$, we see that   $\frac{A}{t+1} \Vert \tau_L (t) \Vert _{L^2} ^2 \leq \Vert \tau_L (t) \Vert _{L^2} ^2$. Multiplying   \eqref{eqn:fourier-splitting-before-integrating-factor}
by   the integrating factor $h( t) = C (t + 1) ^A$,  where
$A > \max \left\{1,  \frac{d}{2} + r^{\ast} \left(u_0 \right),  1+  \frac{d}{2} + r^{\ast} \left(\tau_0 \right) \right\}$, 
we  obtain 
\begin{align*}
  \frac{d}{dt} & \left( C (t + 1) ^A  \left( \omega \Vert u_L(t) \Vert _{L^2}  + \frac{1}{2} \Vert \tau_L (t) \Vert _{L^2} ^2  \right) \right)
                 \leq C (t + 1) ^{A - 1} \left(  (1 + t) ^{- \left( \frac{d}{2} + r^{\ast} \left(u_0 \right) \right)} +  (1 + t) ^{- \left( 1+  \frac{d}{2} + r^{\ast} \left(\tau_0 \right) \right)} \right).
\end{align*}
This finally leads to

\begin{displaymath}
  \omega \Vert u_L(t) \Vert _{L^2}^2  + \frac{1}{2} \Vert \tau_L (t) \Vert _{L^2} ^2 \leq C   (1 + t) ^{- \left( \frac{d}{2} + \min \{  r^{\ast} \left(u_0 \right),  1 +  r^{\ast} \left(\tau_0 \right) \} \right) },
  \quad t>0.
  \end{displaymath}
\end{proof}
%\qed

\section{Proof of the Main Theorems} 

\subsection{Proof of Theorem \ref{thm-a-equal-zero}} 

\noindent
\begin{proof} Standard computations lead to the energy inequality

\begin{displaymath}
\frac{d}{dt} \left( \omega \Vert u(t) \Vert _{L^2}^2 + \frac{1}{2} \Vert \tau (t) \Vert _{L^2} ^2 \right) + \Vert \tau (t) \Vert _{L^2} ^2 \leq - 2 \omega (1 - \omega) \Vert \nabla u (t) \Vert _{L^2} ^2,
\end{displaymath}
from which we obtain

\begin{align}
\label{eqn:first-inequality}
  \frac{d}{dt}  \left( \omega \Vert u(t) \Vert _{L^2}^2  + \frac{1}{2} \Vert \tau (t) \Vert _{L^2} ^2 \right)  & + \Vert \tau (t) \Vert _{L^2} ^2 + 2 \omega \, g^2(t) \Vert u(t) \Vert _{L^2} ^2  \notag  \\
  &\leq  2 \omega \, g^2(t) \int _{B(t)}  |\widehat{u} (\xi, t) | ^2 \, d \xi \notag \\ &  =:2 C  \omega \, g^2(t) \int _{B(t)} \left(  | \widehat{L} (\xi, t) | ^2 + |\widehat{NL} (\xi,t)| ^2 \right) \, d \xi.
\end{align}
Here $B(t) \subset \RR^3$ is defined as in Proposition \ref{decay-linear-part} and $\widehat{L} (\xi, t)$ is as in Proposition \ref{zi-fang-zhang}. As

\begin{align*}
\begin{pmatrix} \widehat{u} (t) \\ \widehat{\tau} (t) \end{pmatrix} & = \mathcal{\widehat{G}} \begin{pmatrix} \widehat{u}_0 \\ \widehat{\tau}_0 \end{pmatrix} - \int _0 ^t \widehat{\mathcal{G}} (t - s) \begin{pmatrix} \PP \, div (u \otimes u) \notag \\ div(u \otimes \tau ) + g_a (\tau, \nabla u)  \end{pmatrix} ^{\wedge} \, ds  = \widehat{L} (\xi, t) + \widehat{NL} (\xi,t), \notag
\end{align*}
then
\begin{align}
  \left( \widehat{u}  \right) _j  (\xi,t) & = \mathcal {B} (\xi,t) \left( \widehat{u_0} \right)_j (\xi) + \mathcal{A} (\xi, t) \, i \, \xi_l \left(\delta _{jk} - \frac{\xi _j \xi_k}{|\xi|^2}\right)
    \left(\widehat{\tau _0} \right)_{lk} (\xi) \notag  \\ &- \int_0 ^t \mathcal{B} (\xi,  t - s) i \,  \xi _l \left( \delta _{jk} - \frac{\xi _j \xi_k}{|\xi|^2} \right) \widehat{u_l u_k} (\xi, s) \, ds \notag \\ & - \int_0 ^t \mathcal{A} (\xi,  t - s) i \,  \xi _l \left( \delta _{jk} - \frac{\xi _j \xi_k}{|\xi|^2} \right) \left( \widehat{div \, (u \otimes \tau)} + \widehat{g_a (\tau, \nabla u)} \right) (\xi, s) \, ds. \notag
\end{align}
The first two terms describe the components of  the solution to the linear problem \eqref{eqn:linear-part-oldroyd}. Hence, Proposition \ref{decay-linear-part} yields 

\begin{equation}
\label{eqn:linear-part-oldroyd-ball}
\int_{B(t)} | \widehat{L} (\xi, t)|^2 \, d \xi \leq   C   (1 + t) ^{- \left( \frac{3}{2} + \min \{  r^{\ast} \left(u_0 \right),  1 +  r^{\ast} \left(\tau_0 \right) \} \right) }, \quad t>0.
\end{equation}
Taking  $|\xi| \leq g(t)  = R$, from Proposition \ref{hieber-wen-zi} we have

\begin{equation}
\label{eqn:pointwise-AB}
|\mathcal{A} (\xi, t)|  \leq C_1 e ^{- \frac{1}{4}(1 - \omega)  |\xi|^2 t}, \, \,  |\mathcal{B} (\xi, t)|  \leq C_2 e ^{- \frac{1}{4}(1 - \omega) |\xi|^2 t}, \quad t>0, 
\end{equation}
for some constants $C_i = C_i (\omega, R), i = 1,2$.  As

\begin{align*}
  \widehat{NL} (\xi,t)  & = - \int_0 ^t \mathcal{B} (\xi,  t - s) i \,  \xi _l \left( \delta _{jk} - \frac{\xi _j \xi_k}{|\xi|^2} \right)  \widehat{u_l u_k} (\xi, s) \, ds  -
  \int_0 ^t \mathcal{A} (\xi,  t - s) i \,  \xi _l \left( \delta _{jk} - \frac{\xi _j \xi_k}{|\xi|^2} \right) \widehat{div \, (u \otimes \tau)} (\xi, s) \, ds \\ & -
  \int_0 ^t \mathcal{A} (\xi,  t - s) i \,  \xi _l \left( \delta _{jk} - \frac{\xi _j \xi_k}{|\xi|^2} \right) \widehat{g_a (\tau, \nabla u)} (\xi, s) \, ds,
\end{align*}
and as we also have 
\begin{align}
\label{eqn:uu-grad-utau}
|\widehat{u_l u_k} (\xi,t)| & \leq C \Vert u(t) \Vert _{L^2} ^2, \,  \quad | \widehat{div \, (u \otimes \tau)} (\xi, t)| \leq C |\xi| \Vert u(t) \Vert _{L^2} \,  \Vert \tau(t)  \Vert _{L^2}, \notag \\  &   |\widehat{g_a (\tau, \nabla u)} (\xi, s)| \leq C \Vert \tau(t) \Vert _{L^2} \Vert \nabla u (t) \Vert _{L^2} \leq C \left( \Vert \tau(t) \Vert _{L^2} ^2 + \Vert \nabla u (t) \Vert _{L^2} ^2 \right),
\end{align}
estimates  \eqref{eqn:pointwise-AB} and \eqref{eqn:uu-grad-utau} imply

\begin{align}
\label{eqn:estimate-nonlinear-term}
\int_{B(t)} |\widehat{NL}(\xi,t)|^2 \, d \xi & \leq C \int_{B(t)} \left( \int_0 ^t |\xi| \Vert u(s) \Vert _{L^2} ^2 \, ds  \right)^2 \, d \xi  + C \int _{B(t)}  \left( \int_0 ^t |\xi|^2 \Vert u(s) \Vert _{L^2} \,  \Vert \tau(s)  \Vert _{L^2} \, ds  \right)^2   \, d \xi \notag \\ & + C \int_{B(t)} \left( \int_0 ^t |\xi| \left( \Vert \tau(t) \Vert _{L^2} ^2 + \Vert \nabla u (t) \Vert _{L^2} ^2\right) \, ds  \right)^2  d \xi \notag \\ & \leq C \int_{B(t)} t |\xi|^2 \left( \int_0 ^t \Vert u(s) \Vert _{L^2} ^4 \, ds \right)  \, d \xi +  C \int _{B(t)} t |\xi|^4  \left( \int_0 ^t  \Vert u(s) \Vert _{L^2} ^2 \,  \Vert \tau(s)  \Vert _{L^2} ^2 \, ds  \right)   \, d \xi \notag \\ & + C \int_{B(t)} t |\xi|^2 \left( \int_0 ^t \Vert \tau (s) \Vert _{L^2} ^4 \, ds \right)  \, d \xi + C \int_{B(t)} |\xi|^2 \left( \int_0 ^t \Vert \nabla u (s) \Vert _{L^2} ^2\right) ^2 \, d \xi \notag \\  & \leq C t \, g^2(t)  \int_{B(t)}  \left( \int_0 ^t \Vert u(s) \Vert _{L^2} ^4 \, ds \right)  \, d \xi + C t \, g^4(t) \int _{B(t)}  \left( \int_0 ^t  \Vert u(s) \Vert _{L^2} ^2 \,  \Vert \tau(s)  \Vert _{L^2} ^2 \, ds  \right)   \, d \xi \notag \\ & + C t \, g^2(t)  \int_{B(t)}  \left( \int_0 ^t \Vert \tau (s) \Vert _{L^2} ^4 \, ds \right)  \, d \xi + C \int_{B(t)} |\xi|^2 \left( \int_0 ^t \Vert \nabla u (s) \Vert _{L^2} ^2\right) ^2 \, d \xi.
\end{align}
Setting  $g(t):= \left(2 C \omega (1 + t) \right) ^{- \frac{1}{2}}$ and using the a priori estimates $\Vert u(t) \Vert _{L^2} \leq C$ and $\Vert \tau (t) \Vert _{L^2} \leq C$, the fact that $u \in L^2 _t \dot{H}^1 _x$ (see Lions and Masmoudi \cite{MR1763488}) and using ${\text Vol} \, B(t) = C (1+t) ^{- \frac{3}{2}}$ in   \eqref{eqn:estimate-nonlinear-term}, we obtain

\begin{displaymath}
\int_{B(t)} |\widehat{NL}(\xi,t)|^2 \, d \xi \leq C (1+t) ^{- \frac{1}{2}} + C (1+t) ^{- \frac{3}{2}}  C (1+t) ^{- \frac{5}{2}} \leq C (1+t) ^{- \frac{1}{2}}, \quad t>0. 
\end{displaymath}
Combining this with \eqref{eqn:first-inequality} and \eqref{eqn:linear-part-oldroyd-ball} yields 

\begin{align}
\frac{d}{dt} & \left( \omega \Vert u(t) \Vert _{L^2}^2  + \frac{1}{2} \Vert \tau (t) \Vert _{L^2} ^2 \right)  + \Vert \tau (t) \Vert _{L^2} ^2 + C (1 + t) ^{-1} \Vert u(t) \Vert _{L^2} ^2  \notag \\ & \leq C (1 + t) ^{-1} \left(    (1 + t) ^{- \left( \frac{3}{2} + \min \{  r^{\ast} \left(u_0 \right),  1 +  r^{\ast} \left(\tau_0 \right) \} \right) } + C (1+t) ^{- \frac{1}{2}} \right), \quad t>0.  \notag
\end{align}
Proceeding as in the proof of Proposition \ref{decay-linear-part}, we multiply by the integrating factor $h(t) = C (1 + t) ^A$, where $A > \max  \left\{1,  \frac{3}{2} + \min \{  r^{\ast} \left(u_0 \right),  1 +  r^{\ast} \left(\tau_0 \right) \} \right\}$. We then obtain

\begin{align}
\frac{d}{dt}  & \left( C (1+t) ^A \left( \omega \Vert u(t) \Vert _{L^2}^2  + \frac{1}{2} \Vert \tau (t) \Vert _{L^2} ^2 \right)  \right) \notag &  \leq C (1 + t) ^{A-1} \left(    (1 + t) ^{- \left( \frac{3}{2} + \min \{  r^{\ast} \left(u_0 \right),  1 +  r^{\ast} \left(\tau_0 \right) \} \right) } + C (1+t) ^{- \frac{1}{2}} \right), \notag
\end{align}
from which we obtain the preliminary estimate

\begin{equation}
\label{eqn:preliminary-estimate}
\omega \Vert u(t) \Vert _{L^2}^2  + \frac{1}{2} \Vert \tau (t) \Vert _{L^2} ^2   \leq C   (1 + t) ^{- \left( \frac{3}{2} + \min \{  r^{\ast} \left(u_0 \right),
    1 +  r^{\ast} \left(\tau_0 \right) \} \right) } + C (1+t) ^{- \frac{1}{2}}, \quad t>0. 
\end{equation}

We now bootstrap  this decay rate. If the slower term on the right-hand side is the second one,  i.e.

\begin{displaymath}
\omega \Vert u(t) \Vert _{L^2}^2  + \frac{1}{2} \Vert \tau (t) \Vert _{L^2} ^2   \leq  C (1+t) ^{- \frac{1}{2}}, \quad t>0,
\end{displaymath}
then 

\begin{align*}
%\label{eqn:end-of-loop}
 \int_0 ^t \Vert u(s) \Vert _{L^2} ^4 \, ds & \leq C \ln (1+t) \leq C (1+t)^{\mu}, \quad \int_0 ^t  \Vert u(s) \Vert _{L^2} ^2 \,  \Vert \tau(s)  \Vert _{L^2} ^2 \, ds \leq C \ln (1+t) \leq C (1+t)^{\mu} \notag \\ &  \int_0 ^t \Vert \tau (s) \Vert _{L^2} ^4 \, ds  \leq C \ln (1+t) \leq C (1+t)^{\mu},
\end{align*}
for some $0 < \mu < 1$. From \eqref{eqn:estimate-nonlinear-term} we obtain

\begin{equation*}
%\label{eqn:decay-nonlinear-term}
\int_{B(t)} |\widehat{NL}(\xi,t)|^2 \, d \xi  \leq C (1+t) ^{- \frac{3}{2} + \mu}, \quad t>0.
\end{equation*}
Proceeding as before we get

\begin{displaymath}
  \omega \Vert u(t) \Vert _{L^2}^2  + \frac{1}{2} \Vert \tau (t) \Vert _{L^2} ^2   \leq C   (1 + t) ^{- \left( \frac{3}{2} + \min \{  r^{\ast} \left(u_0 \right), 1 +  r^{\ast} \left(\tau_0 \right) \} \right) } + C (1+t) ^{- \frac{3}{2} + \mu}, \quad t>0. 
\end{displaymath}
Suppose that again the slower term is the last one on the right hand side. Note that as $-\frac{3}{2} + \mu < - \frac{1}{2}$, then 

\begin{equation}
\label{eqn:end-of-loop}
 \int_0 ^t \Vert u(s) \Vert _{L^2} ^4 \, ds \leq C, \, \, \, \,     \int_0 ^t  \Vert u(s) \Vert _{L^2} ^2 \,  \Vert \tau(s)  \Vert _{L^2} ^2 \, ds \leq C, \, \, \,  \, \int_0 ^t \Vert \tau (s) \Vert _{L^2} ^4 \, ds \leq C.
\end{equation}
Going back to \eqref{eqn:estimate-nonlinear-term} and then following the same argument we obtain
\begin{displaymath}
  \omega \Vert u(t) \Vert _{L^2}^2  + \frac{1}{2} \Vert \tau (t) \Vert _{L^2} ^2   \leq C   (1 + t) ^{- \left( \frac{3}{2} + \min \{  r^{\ast} \left(u_0 \right),
      1 +  r^{\ast} \left(\tau_0 \right) \} \right) } + C (1+t) ^{- \frac{3}{2}}, \quad t>0.
\end{displaymath}
If  the slower term were still the second one,  we circle back to \eqref{eqn:end-of-loop} and no improvement is obtained.

Now suppose in \eqref{eqn:preliminary-estimate} the slower term were the first one, i.e.

\begin{equation*}
\gamma :=  - \left( \frac{3}{2} + \min \{  r^{\ast} \left(u_0 \right), 1 +  r^{\ast} \left(\tau_0 \right) \} \right) > - \frac{1}{2}.
\end{equation*}
This leads to

\begin{align*}
 \int_0 ^t \Vert u(s) \Vert _{L^2} ^4 \, ds  & \leq C (1+t)^{2 \gamma + 1},\quad  \int_0 ^t  \Vert u(s) \Vert _{L^2} ^2 \,  \Vert \tau(s)  \Vert _{L^2} ^2 \, ds\leq C (1+t)^{2 \gamma + 1} \notag \\ & \int_0 ^t \Vert \tau (s) \Vert _{L^2} ^4 \, ds   \leq C (1+t)^{2 \gamma + 1}, 
\end{align*}
where $0 < 2 \gamma + 1 < 1$. Hence

\begin{displaymath}
\int_{B(t)} |\widehat{NL}(\xi,t)|^2 \, d \xi  \leq C (1+t) ^{- \frac{1}{2} + 2 \gamma}, \quad t>0,
\end{displaymath}
so 

\begin{align*}
  \omega \Vert u(t) \Vert _{L^2}^2  + \frac{1}{2} \Vert \tau (t) \Vert _{L^2} ^2   & \leq C   (1 + t) ^{- \left( \frac{3}{2} + \min \{  r^{\ast} \left(u_0 \right), 1 +  r^{\ast} \left(\tau_0 \right) \} \right) } + C (1+t) ^{- \frac{1}{2} + 2 \gamma}, \quad t>0 \\ & \leq C   (1 + t) ^{- \left( \frac{3}{2} + \min \{  r^{\ast} \left(u_0 \right), 1 +  r^{\ast} \left(\tau_0 \right) \} \right) }, 
\end{align*}
because $2 \gamma  < 0$. So, our estimate stays the same. 

If at any time during the iterations in the first case the slower term were the one coming from the linear part instead of that arising from the nonlinear part, we would switch to the latter part of the argument, as just described.  
  We have thus established our  result. 

\end{proof}

\subsection{Proof of Theorem \ref{general-a}} 

\begin{proof}
We start by  proving \eqref{eqn:higher-derivatives} for $k = 0$.  As in Hieber, Wen and Zi \cite{MR3909068}, we obtain the estimate

\begin{equation}
\label{eqn:equation59}
\frac{d}{dt} F_1 (t) +  \frac{1}{2} F_2 (t)   \leq \frac{2 \kappa_2}{(1 - \omega) (1 + t)} \int _{B(t)} |\widehat{u} (\xi,t)| ^2 \, d \xi,
\end{equation}
where 

\begin{equation*}
    B(t) = \left\{ \xi \in \RR^3: |\xi| \leq  \left( \frac{4 \kappa_2}{(1 - \omega)(1 + t)}\right) ^{\frac{1}{2}} \right\},
\end{equation*}
and
\begin{align}
\label{eqn:f1}
  F_1 (t) &:= \Vert \tau(t) \Vert _{H^2} ^2 + (1 - \omega) \kappa_1 \Vert \mbox{curl} \, \mbox{div} \tau (t) \Vert _{L^2} ^2 + \frac{\kappa_2}{2 \omega (1 - \omega)}
            \Vert \partial_t \tau (t) \Vert _{L^2} ^2 \notag \\ & \quad + (1 - \omega) \Vert \nabla u(t) \Vert _{L^2} ^2 +   \frac{\kappa_2}{1 - \omega} \Vert u(t) \Vert _{L^2} ^2
         + \frac{\kappa_2}{1 - \omega} \Vert \partial_t u(t) \Vert _{L^2} ^2,
\end{align} and 
\begin{align*}
F_2 (t) &:= \Vert \tau(t) \Vert _{H^2} ^2 + \Vert \mbox{curl} \, \mbox{div} \tau (t) \Vert _{L^2} ^2 + \Vert \partial_t \tau (t) \Vert _{L^2} ^2 \notag \\ & \quad + \Vert \nabla u(t) \Vert _{L^2} ^2 +   \frac{ 2
\kappa_2}{(1 - \omega) (1 + t)} \Vert u(t) \Vert _{L^2} ^2   + \Vert \partial_t u(t) \Vert _{L^2} ^2,
\end{align*}
for some constants $\kappa_1, \kappa_2 > 0$. Then 
\begin{equation}
    \label{eqn:integrating-factor}
    \frac{d}{dt} \left((1+t) F_1(t) \right) \leq (1+t) \frac{d}{dt} F_1 (t) + \frac{1}{2} (1+t) F_2 (t),  
\end{equation}
for $1 + t > \max \left\{2, 2(1-\omega)\kappa_1, \frac{\kappa_2}{\omega (1 - \omega)}, \frac{\kappa_2}{1 - \omega} \right\}$. Hence, the estimate will be valid for $t > T = T(\omega)$, where

\begin{displaymath}
    T = T(\omega) = \max \left\{2, 2(1-\omega)\kappa_1, \frac{\kappa_2}{\omega (1 - \omega)}, \frac{\kappa_2}{1 - \omega} \right\} -1.
\end{displaymath}

We are now aiming for a pointwise estimate for $|\widehat{u} (\xi,t)|$ in $B(t)$ in \eqref{eqn:equation59}. Compared to the case $a = 0$ from Theorem \ref{thm-a-equal-zero}, now the term $g_a$ is 

\begin{displaymath}
    g_a (\tau, \nabla u) = \tau W(u) - W(u) \tau - a \left(D(u) \tau + \tau D(u) \right).
\end{displaymath}
However, the fact that 

\begin{displaymath}
    |\widehat{\left( (D(u) \tau \right)}(\xi,t)| \leq C \Vert \nabla u (t) \Vert _{L^2} \Vert \tau (t) \Vert _{L^2},
\end{displaymath}
implies that $g_a$ still obeys the same bound as in the case $a = 0$, i.e. \eqref{eqn:uu-grad-utau}. Then, the proof for our estimate is identical to that of Theorem \ref{thm-a-equal-zero}. In the end we obtain

\begin{equation}
\label{eqn:estimate-k-equal-zero}
\Vert u(t) \Vert _{L^2} ^2  + \Vert \tau(t) \Vert _{L^2} ^2 \leq C  (1+t) ^{- \left( \frac{3}{2} + \min \left\{  r^{\ast} \left(u_0 \right),  1 +  r^{\ast} \left(\tau_0 \right) \right\} \right)}, \quad, t>0, 
\end{equation}
hence our result holds for $k=0$.

Let us continue with the case $k = 1$.  Considering  the ball

\begin{displaymath}
   B(t) = \left\{\xi \in \RR^3: |\xi|^2 \leq g^2(t) = \frac{24}{(1 - \omega)} (1 + t) ^{- \frac{1}{2}} \right\} 
\end{displaymath}
and 
following the computations in Hieber, Wen and Zi \cite{MR3909068},  we obtain

\begin{equation}
\label{eqn:inequality-g}
\frac{d}{dt} G(t) + \frac{3}{1 + t} G(t) \leq \frac{3 \omega}{2(1+t)} \int _{B(t)} |\xi|^2 |\widehat{u}(\xi,t)|^2 \, d \xi + \Vert \nabla u(t) \Vert _{L^2} ^6, \qquad t \geq t_1,
\end{equation}
for some large $t_1 > 0$, where
\begin{align}
  G(t) & = 2 \delta \Vert \nabla \tau (t) \Vert _{L^2} ^2 + \frac{1}{2} \Vert div \, \tau (t) \Vert _{L^2} ^2 + \frac{9 \omega^2}{\omega (1 - \omega)} \Vert \nabla \, div \tau (t) \Vert _{L^2} ^2
         + \Vert \nabla ^2 \tau (t) \Vert _{L^2} ^2 \notag \\ & + \frac{\omega}{2} \Vert \nabla u (t) \Vert _{L^2} ^2 + \frac{9 \omega^3}{\omega (1 - \omega)} \Vert \nabla ^2 u (t) \Vert _{L^2} ^2, \notag
\end{align}
where $\delta > 0$ is a constant. The pointwise estimate in the ball $B(t)$ already obtained leads to

\begin{equation}
\label{eqn:xi-squared-ball}
\int _{B(t)} |\xi|^2 |\widehat{u}(\xi,t)|^2 \, d \xi \leq g^2 (t) \int _{B(t)} |\widehat{u}(\xi,t)|^2 \, d \xi \leq (1 + t) ^{- \left( 1 + \min \left\{   \frac{3}{2} + \min \{  r^{\ast} \left(u_0 \right),  1 +  r^{\ast} \left(\tau_0 \right) \}, \frac{3}{2}
    \right\} \right)},\,\, t>0, 
\end{equation}
and as a  consequence  of the fact that $\Vert \nabla u (t) \Vert _{L^2}$ also obeys the estimate \eqref{eqn:estimate-k-equal-zero},  because of \eqref{eqn:equation59}, \eqref{eqn:f1} and \eqref{eqn:integrating-factor}, we obtain

\begin{equation}
\label{eqn:nabla-six}
\Vert \nabla u(t) \Vert _{L^2} ^6 \leq C (1 + t) ^{- 3 \min \left\{   \frac{3}{2} + \min \{  r^{\ast} \left(u_0 \right),  1 +  r^{\ast} \left(\tau_0 \right) \}, \frac{3}{2}  \right\}}, \quad t>0.
\end{equation}
Using the integrating factor $h(t) = (1 + t)^A$, with 

\begin{displaymath}
    A > \max \left\{ 1 + \min \left\{   \frac{3}{2} + \min \left\{  r^{\ast} \left(u_0 \right), 1 + r^{\ast} \left(\tau_0 \right) \right\}, \frac{3}{2} \right\}, 3 \min \left\{   \frac{3}{2} + \min \{  r^{\ast} \left(u_0 \right),  1 +  r^{\ast} \left(\tau_0 \right) \}, \frac{3}{2}  \right\} - 1 \right\}
\end{displaymath}
through \eqref{eqn:inequality-g} we get

\begin{align}
\frac{d}{dt} \left( (1+t)^A G(t) \right) & \leq  (1+t)^{A-1} (1 + t) ^{- \left( 1 + \min \left\{   \frac{3}{2} + \min \{  r^{\ast} \left(u_0 \right),  1 +  r^{\ast} \left(\tau_0 \right) \}, \frac{3}{2}  \right\} \right)} \notag \\ & + (1+t)^A (1 + t) ^{- 3 \min \left\{   \frac{3}{2} + \min \{  r^{\ast} \left(u_0 \right),  1 +  r^{\ast} \left(\tau_0 \right) \}, \frac{3}{2}  \right\}} \notag 
\end{align}
Integrating we obtain

\begin{align}
\label{eqn:estimate-g}
2 \delta \Vert \nabla \tau (t) \Vert _{L^2} ^2 & + \frac{1}{2} \Vert div \, \tau (t) \Vert _{L^2} ^2 + \frac{9 \omega^2}{\omega (1 - \omega)} \Vert \nabla \, div \tau (t) \Vert _{L^2} ^2 + \Vert \nabla ^2 \tau (t) \Vert _{L^2} ^2 \notag \\ & + \frac{\omega}{2} \Vert \nabla u (t) \Vert _{L^2} ^2 + \frac{9 \omega^3}{\omega (1 - \omega)} \Vert \nabla ^2 u (t) \Vert _{L^2} ^2 \notag \\ & \leq C (1 + t) ^{- \left( 1 + \min \left\{   \frac{3}{2} + \min \{  r^{\ast} \left(u_0 \right),  1 +  r^{\ast} \left(\tau_0 \right) \}, \frac{3}{2}  \right\} \right)} \notag \\ & +  (1 + t) ^{-\left( 3  \min \left\{   \frac{3}{2} + \min \{  r^{\ast} \left(u_0 \right),  1 +  r^{\ast} \left(\tau_0 \right) \}, \frac{3}{2}  \right\} - 1 \right)}, \quad t>0.
\end{align}
Suppose that the first term in the right hand side of \eqref{eqn:estimate-g} is slower. Then \eqref{eqn:nabla-six} becomes

\begin{displaymath}
\Vert \nabla u(t) \Vert _{L^2} ^6 \leq C (1 + t) ^{- 3 \left(1 +  \min \left\{   \frac{3}{2} + \min \{  r^{\ast} \left(u_0 \right),  1 +  r^{\ast} \left(\tau_0 \right) \}, \frac{3}{2}  \right\} \right)}, \quad t>0,
\end{displaymath}
which leads, following the computations above, to 

\begin{align}
\label{eqn:estimate-first-derivative}
2 \delta \Vert \nabla \tau (t) \Vert _{L^2} ^2 & + \frac{1}{2} \Vert div \, \tau (t) \Vert _{L^2} ^2 + \frac{9 \omega^2}{\omega (1 - \omega)} \Vert \nabla \, div \tau (t) \Vert _{L^2} ^2 + \Vert \nabla ^2 \tau (t) \Vert _{L^2} ^2 \notag \\ & + \frac{\omega}{2} \Vert \nabla u (t) \Vert _{L^2} ^2 + \frac{9 \omega^3}{\omega (1 - \omega)} \Vert \nabla ^2 u (t) \Vert _{L^2} ^2 \notag \\ & \leq C (1 + t) ^{- \left( 1 + \min \left\{   \frac{3}{2} + \min \{  r^{\ast} \left(u_0 \right),  1 +  r^{\ast} \left(\tau_0 \right) \}, \frac{3}{2}  \right\} \right)}, \quad t>0.
\end{align}
Further iterations with the same condition lead to no improvement in decay. Now, if the second term in \eqref{eqn:estimate-g} were slower, after a finite number of bootstrap steps the estimate for $\Vert \nabla u(t) \Vert _{L^2} ^6$ would become smaller than the other term in \eqref{eqn:estimate-g}, so we would be in the previouse case again. Hence, we have proved \eqref{eqn:higher-derivatives} for $k = 1$. 

It remains to prove the case of $k=2$. From Hieber, Wen and Zi \cite{MR3909068} we have that

\begin{equation}
\label{eqn:eq-h}
    \frac{d}{dt} H(t) + \frac{4}{1+t} H(t) \leq \frac{36 \, \omega ^3}{(1 - \omega) \omega (1+t)} \int _{B(t)} |\xi|^4 |\widehat{u} (\xi,t)|^2 \, d \xi + C \Vert \nabla ^2 u(t) \Vert _{L^2} ^{\frac{10}{3}},  \quad t \geq t_2
\end{equation}
for some large enough $t_2 > 0$, where 
\begin{displaymath}
  H (t)  = \frac{9 \omega^3}{\omega (1 - \omega)} \Vert \nabla ^2 u \Vert _{L^2} ^2 + \frac{9 \omega^2}{\omega (1 - \omega)} \Vert \nabla div \tau (t) \Vert _{L^2} ^2 + \Vert \nabla ^2 \tau (t)
  \Vert _{L^2} ^2, 
\end{displaymath}
and 
\begin{displaymath}
    B(t) = \left\{\xi \in \RR^3: |\xi|^2 \leq \frac{72 \omega}{\omega (1 - \omega)} (1 + t) ^{- \frac{1}{2}} \right\}.
\end{displaymath}
An argument analogous to that in \eqref{eqn:xi-squared-ball} leads to

\begin{equation*}
%\label{eqn:xi-fourth-ball}
\int _{B(t)} |\xi|^4 |\widehat{u}(\xi,t)|^2 \, d \xi \leq (1 + t) ^{- \left( 2 + \min \left\{   \frac{3}{2} + \min \{  r^{\ast} \left(u_0 \right),
      1 +  r^{\ast} \left(\tau_0 \right) \}, \frac{3}{2}  \right\} \right)}, \quad t>0. 
\end{equation*}
Next,  \eqref{eqn:estimate-first-derivative} implies

\begin{displaymath}
  \Vert \nabla ^2 u(t) \Vert _{L^2} ^{\frac{10}{3}} \leq (1 + t) ^{- \frac{5}{3} \left( 1 + \min \left\{   \frac{3}{2} + \min \{  r^{\ast} \left(u_0 \right),
        1 +  r^{\ast} \left(\tau_0 \right) \}, \frac{3}{2}  \right\} \right)}, \quad t>0.
\end{displaymath}
Using the  integrating factor $k(t) = (1+t)^A$ in \eqref{eqn:eq-h} with

\begin{displaymath}
    A > \max \left\{ 2 + \min \left\{   \frac{3}{2} + \min \{  r^{\ast} \left(u_0 \right),
      1 +  r^{\ast} \left(\tau_0 \right) \}, \frac{3}{2}  \right\}, \frac{5}{3} \min \left\{   \frac{3}{2} + \min \{  r^{\ast} \left(u_0 \right),  1 +  r^{\ast} \left(\tau_0 \right) \}, \frac{3}{2}  \right\} - 1 \right\}
\end{displaymath}
yields

\begin{align}
\frac{d}{dt} \left( (1+t)^A H(t) \right) & \leq  (1+t)^{A-1} (1 + t) ^{- \left( 2 + \min \left\{   \frac{3}{2} + \min \{  r^{\ast} \left(u_0 \right),  1 +  r^{\ast} \left(\tau_0 \right) \}, \frac{3}{2}  \right\} \right)} \notag \\ & + (1+t)^A (1 + t) ^{- \frac{5}{3} \min \left\{   \frac{3}{2} + \min \{  r^{\ast} \left(u_0 \right),  1 +  r^{\ast} \left(\tau_0 \right) \}, \frac{3}{2}  \right\}}, \quad t>0. \notag
\end{align}
Integrating gives

\begin{align}
    \frac{9 \omega^3}{\omega (1 - \omega)} \Vert \nabla ^2 u \Vert _{L^2} ^2 & + \frac{9 \omega^2}{\omega (1 - \omega)} \Vert \nabla div \tau (t) \Vert _{L^2} ^2 + \Vert \nabla ^2 \tau (t) \Vert _{L^2} ^2 \notag \\ & \leq  (1 + t) ^{- \left( 2 + \min \left\{   \frac{3}{2} + \min \{  r^{\ast} \left(u_0 \right),  1 +  r^{\ast} \left(\tau_0 \right) \}, \frac{3}{2}  \right\} \right)} \notag \\ & +  (1 + t) ^{- \left( \frac{5}{3} \min \left\{   \frac{3}{2} + \min \{  r^{\ast} \left(u_0 \right),  1 +  r^{\ast} \left(\tau_0 \right) \}, \frac{3}{2}  \right\} - 1 \right)}. \notag
\end{align}
An argument analogous to that used in the case $k = 1$ leads to 

\begin{align*}
    \frac{9 \omega^3}{\omega (1 - \omega)} \Vert \nabla ^2 u \Vert _{L^2} ^2 & + \frac{9 \omega^2}{\omega (1 - \omega)} \Vert \nabla div \tau (t) \Vert _{L^2} ^2 + \Vert \nabla ^2 \tau (t) \Vert _{L^2} ^2 \notag \\ & \leq  (1 + t) ^{- \left( 2 + \min \left\{   \frac{3}{2} + \min \{  r^{\ast} \left(u_0 \right),  1 +  r^{\ast} \left(\tau_0 \right) \}, \frac{3}{2}  \right\} \right)},
\end{align*}
which implies \eqref{eqn:higher-derivatives} for $k =2$.

Finlly, we prove estimate \eqref{eqn:high-deriv-tau}. The equation for $\tau$ in \eqref{eqn:oldroyd-b} implies

\begin{equation*}
%\label{eqn:better-decay-tau-j=0}
    \frac{1}{2} \frac{d}{dt} \Vert \tau (t) \Vert _{L^2} ^2 + \Vert \tau (t) \Vert _{L^2} ^2 = 2 \omega \int_{\RR^3} \tau D (u) \, dx - \int_{\RR^3} \tau \, u \nabla \tau  \, dx - \int_{\RR^3} \tau g_a (\tau, \nabla u) \, dx.
\end{equation*}
Notice first that  
\begin{equation}
\label{eqn:better-decay-tau-first}
\int_{\RR^3} \tau D (u) \, dx \leq \Vert \tau(t) \Vert _{L^2} \Vert \nabla u \Vert _{L^2} \leq \frac{1}{2} \left(\Vert \tau(t) \Vert _{L^2} ^2 +  \Vert \nabla u \Vert _{L^2} ^2 \right).
\end{equation}
Furthermore, 
\begin{equation}
\label{eqn:better-decay-tau-second}
\int_{\RR^3} \tau \, u \nabla \tau  \, dx \leq \Vert \tau (t) \Vert _{L^3} \Vert u (t) \Vert _{L^6} \Vert \nabla \tau (t) \Vert _{L^2} \leq C \Vert \tau(t) \Vert _{L^2} ^{\frac{1}{2}} \Vert D \tau (t)  \Vert _{L^2} ^{\frac{1}{2}} \Vert \nabla u (t) \Vert _{L^2} \Vert \nabla \tau (t) \Vert _{L^2}
\end{equation}
for some $C>0$, where we used 

\begin{equation}
\label{eqn:inequalities}
\Vert f \Vert _{L^6 (\RR^3)} \leq C \Vert \nabla f \Vert _{L^2 (\RR^3)}, \qquad  \Vert f \Vert  _{L^3 (\RR^3)} \leq C \Vert f \Vert _{L^2 (\RR^3)} ^{\frac{1}{2}} \Vert Df \Vert _{L^2 (\RR^3)} ^{\frac{1}{2}}. 
\end{equation}
Finally, the interpolation inequality $\Vert f \Vert _{L^4} \leq C \Vert f \Vert _{L^2} ^{\frac{1}{4}} \Vert \nabla f \Vert _{L^2} ^{\frac{3}{4}}$ implies
\begin{align}
\label{eqn:better-decay-tau-third}
\int_{\RR^3} \tau g_a (\tau, \nabla u) \, dx & \leq C \Vert \tau ^2 (t) \Vert _{L^2} \Vert \nabla u(t) \Vert _{L^2} = C \Vert \tau (t) \Vert _{L^4} ^2 \Vert \nabla u(t) \Vert _{L^2} \leq \frac{C}{2} \left(\Vert \tau (t) \Vert _{L^4} ^4 +  \Vert \nabla u(t) \Vert _{L^2} ^2 \right) \notag \\ & \leq C \Vert \tau (t) \Vert _{L^2} \Vert \nabla \tau (t) \Vert _{L^2} ^3 + C  \Vert \nabla u(t) \Vert _{L^2} ^2
\end{align}
for some $C>0$. Using \eqref{eqn:higher-derivatives} and keeping the terms with slower terms decay in \eqref{eqn:better-decay-tau-first}, \eqref{eqn:better-decay-tau-second} and \eqref{eqn:better-decay-tau-third}, we obtain

\begin{displaymath}
\frac{1}{2} \left( \frac{d}{dt} \Vert \tau (t) \Vert _{L^2} ^2 + \Vert \tau (t) \Vert _{L^2} ^2 \right) \leq C \Vert \nabla u(t) \Vert _{L^2} ^2 = C (1 + t) ^{- \left( 1 + \min \left\{   \frac{3}{2} + \min \{  r^{\ast} \left(u_0 \right),  1 +  r^{\ast} \left(\tau_0 \right) \}, \frac{3}{2}  \right\} \right)}, \quad t>0, 
\end{displaymath}
which after using Gronwall's inequality leads to \eqref{eqn:high-deriv-tau} for $j = 0$.

In order to prove \eqref{eqn:high-deriv-tau} for  $j = 1$,  we take gradient of the equation for $\tau$  and multiply by $\nabla \tau$ to obtain

\begin{equation*}
    \frac{1}{2} \frac{d}{dt} \Vert \nabla \tau (t) \Vert _{L^2} ^2 + \Vert \tau (t) \Vert _{L^2} ^2 = \int_{\RR^3} \nabla \tau \nabla D (u) \, dx - \int_{\RR^3} \nabla \tau \,  \nabla \left( u \nabla \tau \right)  \, dx - \int_{\RR^3} \nabla \tau \, \nabla g_a (\tau, \nabla u) \, dx.
\end{equation*}
Proceeding as before we obtain 

\begin{equation}
\label{eqn:better-decay-tau-j=1-first}
\int_{\RR^3} \nabla \tau \nabla D (u) \, dx \leq \Vert \nabla \tau(t) \Vert _{L^2} \Vert \nabla ^2 u \Vert _{L^2} \leq \frac{1}{2} \left(\Vert \nabla \tau(t) \Vert _{L^2} ^2 +
  \Vert \nabla^2 u \Vert _{L^2} ^2 \right), \quad t>0,
\end{equation}
\begin{align}
\label{eqn:better-decay-tau-j=1-second}
  \left|\int_{\RR^3} \nabla \tau \,  \nabla \left( u \nabla \tau \right)  \, dx \right| & = \left| \int_{\RR^3} \nabla ^2 \tau \,   u \nabla \tau  \, dx \right| \leq \Vert \nabla^2 \tau (t) \Vert _{L^2} \Vert u (t) \Vert _{L^3} \Vert \nabla \tau (t) \Vert _{L^6} \nonumber \\ & \leq C \Vert u(t) \Vert _{L^2} ^{\frac{1}{2}} \Vert D u (t)  \Vert _{L^2} ^{\frac{1}{2}} \Vert \nabla^2 \tau (t) \Vert _{L^2} ^2, \quad t>0, 
\end{align}
where we again used the embedding $\dot{H}^1 (\RR^3) \subset L^6 (\RR^3)$ and  the fact that $\Vert f \Vert  _{L^3 (\RR^3)} \leq C \Vert f \Vert _{L^2 (\RR^3)} ^{\frac{1}{2}} \Vert Df
\Vert _{L^2 (\RR^3)} ^{\frac{1}{2}} $. Now,
\begin{align}
\label{eqn:better-decay-tau-j=1-third}
    \left| \int_{\RR^3} \nabla \tau \, \nabla g_a (\tau, \nabla u) \, dx \right| & = \left| \int_{\RR^3} \nabla^2 \tau \, g_a (\tau, \nabla u) \, dx \right| \leq C \Vert \nabla^2 \tau (t) \Vert _{L^2} \Vert \tau (t) \Vert _{L^3} \Vert \nabla u (t) \Vert _{L^6} \notag \\ & \leq C \Vert \tau(t) \Vert _{L^2} ^{\frac{1}{2}} \Vert D \tau (t)  \Vert _{L^2} ^{\frac{1}{2}} \Vert \nabla^2 \tau (t) \Vert _{L^2} \Vert \nabla^2 u (t) \Vert _{L^2}, 
\end{align}
As before, we use \eqref{eqn:higher-derivatives} and keep the terms with slower terms decay in \eqref{eqn:better-decay-tau-j=1-first}, \eqref{eqn:better-decay-tau-j=1-second} and \eqref{eqn:better-decay-tau-j=1-third} to obtain
\begin{displaymath}
\frac{1}{2} \left( \frac{d}{dt} \Vert \nabla \tau (t) \Vert _{L^2} ^2 + \Vert \nabla \tau (t) \Vert _{L^2} ^2 \right) \leq C \Vert \nabla^2 \tau(t) \Vert _{L^2} ^2 = C (1 + t) ^{2 \left( 1 + \min \left\{   \frac{3}{2} + \min \{  r^{\ast} \left(u_0 \right),  1 +  r^{\ast} \left(\tau_0 \right) \}, \frac{3}{2}  \right\} \right)}, \quad t>0,
\end{displaymath}
which after using Gronwall's inequality leads to \eqref{eqn:high-deriv-tau} for $j = 1$.% i.e.

The proof of Theorem \ref{general-a} is complete.
\end{proof}

\subsection{Proof of Theorem \ref{thm:error}}

\begin{proof}
For notational convenience we set
\begin{displaymath}
\alpha = \alpha(r^\star(u_0),r^\star(\tau_0)) =  \min \left\{\frac{3}{2}, \frac{3}{2} +  \min \{\,  r^{\ast} \left(u_0 \right),  1 + r^{\ast} \left(\tau_0 \right)\, \}\, \right\} .    
\end{displaymath}
Recall that $ \varepsilon = \tau -2 \omega D(u)$. Taking the time derivative and using \eqref{eqn:oldroyd-b} yields

\begin{displaymath}
\partial_t \varepsilon +  \varepsilon - (1 - \omega) \Delta \varepsilon = - (u\cdot\nabla)\tau - g_a(\tau, \nabla u)  - 2 \omega D \, \mathbb{P} div \, \tau + 2 \omega \, D\mathbb{P} (u \cdot \nabla) u - (1 - \omega) \Delta \tau, \quad t>0.
\end{displaymath}
Multiplication by $\varepsilon$ and integration leads to

\begin{align*}
\frac{1}{2} \frac{d}{dt} \Vert \varepsilon (t) \Vert _{L^2} ^2 +   \Vert \varepsilon (t) \Vert _{L^2} ^2 + (1 - \omega) \Vert \nabla \varepsilon (t) \Vert _{L^2} ^2 & = - \int _{\RR^3} \varepsilon \, (u\cdot\nabla)\tau \, dx - \int _{\RR^3} \varepsilon \,  g_a(\tau, \nabla u) \, dx - 2 \omega \int _{\RR^3} \varepsilon \, D \, \mathbb{P} div \, \tau \, dx \notag \\ & + 2 \omega  \int _{\RR^3} \varepsilon \, D\mathbb{P} (u \cdot \nabla) u \, dx - (1 - \omega) \int _{\RR^3} \varepsilon \, \Delta \tau  \, dx. 
\end{align*}
Notice first that
\begin{align} 
\int _{\RR^3} \varepsilon \, (u\cdot\nabla)\tau \, dx & \leq \Vert \varepsilon (t) \Vert _{L^2} \Vert u (t) \Vert _{L^6} \Vert \nabla \tau (t) \Vert _{L^3} \leq \Vert \varepsilon (t) \Vert _{L^2} \Vert \nabla u (t) \Vert _{L^2} \Vert \nabla \tau (t) \Vert _{L^2} ^{1/2} \Vert \nabla ^2 \tau (t)  \Vert _{L^2} ^{1/2} \notag \\ & \leq C (1+t) ^{- \left(\frac{1}{2} + \frac{\alpha}{2} \right)} \, (1+t) ^{- \left(\frac{1}{2} + \frac{\alpha}{2} \right)} \, (1+t) ^{- \left(\frac{1}{2} + \frac{\alpha}{4} \right)} \, (1+t) ^{- \left(\frac{1}{2} + \frac{\alpha}{4} \right)}  \notag \\ & = C (1+t) ^{- \left( 2 +  \frac{3 \alpha}{2} \right)}, \quad t>0,\notag
\end{align} 
where we have used \eqref{eqn:inequalities} and the decay estimates proved in Theorem \ref{general-a}. Recalling that $g_a$ is a sum of products of $\tau$ and $\nabla u$, we verify that 

\begin{align}
    \int _{\RR^3} \varepsilon \,  g_a(\tau, \nabla u) \, dx & \leq C \Vert \varepsilon (t) \Vert _{L^2} \Vert \tau (t) \Vert _{L^6} \Vert \nabla u(t) \Vert _{L^3} \leq C \Vert \varepsilon (t) \Vert _{L^2} \Vert \nabla \tau (t) \Vert _{L^2} \Vert \nabla u (t) \Vert _{L^2} ^{1/2} \Vert \nabla ^2 u (t)  \Vert _{L^2} ^{1/2} \notag \\ & \leq C (1+t) ^{- \left(\frac{1}{2} + \frac{\alpha}{2} \right)} \, C (1+t) ^{- \left(1 + \frac{\alpha}{2} \right)} \, C (1+t) ^{- \left(\frac{1}{4} + \frac{\alpha}{4} \right)} \, C (1+t) ^{- \left(\frac{1}{2} + \frac{\alpha}{4} \right)} \notag \\ & = 
C (1+t) ^{- \left( \frac{9}{4} + \frac{3 \alpha}{2} \right)}, \quad t>0,\notag
\end{align}
where we used  \eqref{eqn:inequalities} again, as well as the decay estimates from Theorem \ref{general-a}. Noting  that  

\begin{displaymath}
    \int _{\RR^3} \varepsilon \, D \, \mathbb{P} div \, \tau \, dx = - \int _{\RR^3} D \varepsilon \, \mathbb{P} div \, \tau \, dx,
\end{displaymath}
we have 
\begin{displaymath}
%\label{eqn:slow1}
    \int _{\RR^3} \varepsilon \, D\mathbb{P} div \, \tau \, dx \leq C \Vert D \varepsilon (t) \Vert _{L^2} \Vert div \, \tau (t) \Vert _{L^2} \leq C (1+ t) ^{- \left( 2 + \alpha \right)}, \quad t>0.
\end{displaymath}
Next, we observe that 

\begin{displaymath}
\int _{\RR^3} \varepsilon \, D\mathbb{P} (u \cdot \nabla) u \, dx = - \int _{\RR^3} D \varepsilon \, \mathbb{P} (u \cdot \nabla) u \, dx.
\end{displaymath}
Then, we obtain 
\begin{align}
\int _{\RR^3} D \varepsilon \, \mathbb{P} (u \cdot \nabla) u \, dx & \leq C \Vert D \varepsilon (t) \Vert _{L^2} \Vert u (t) \Vert _{L^6} \Vert \nabla u(t) \Vert _{L^3} \leq C \Vert D \varepsilon (t) \Vert _{L^2} \Vert  \nabla u (t) \Vert _{L^2}  \Vert \nabla u (t) \Vert _{L^2} ^{1/2} \Vert \nabla ^2 u (t)  \Vert _{L^2} ^{1/2} \notag \\ & \leq C (1+t) ^{- \left(1 + \frac{\alpha}{2} \right)} \, (1+t) ^{- \left(\frac{1}{2} + \frac{\alpha}{2} \right)} \, (1+t) ^{- \left(\frac{1}{4} + \frac{\alpha}{4} \right)} \, (1+t) ^{- \left(\frac{1}{2} + \frac{\alpha}{4} \right)} = C (1+t) ^{- \left(\frac{9}{4} + \frac{3 \alpha}{2} \right)}, \quad t>0.\notag
\end{align}
Finally, since 

\begin{displaymath}
    \int _{\RR^3} \varepsilon \, \Delta \tau  \, dx  = - \int _{\RR^3} \nabla \varepsilon \, \nabla \tau  \, dx,
\end{displaymath}
we see   that
\begin{equation*}
    \int _{\RR^3} \nabla \varepsilon \, \nabla \tau  \, dx  \leq \Vert \nabla \varepsilon (t) \Vert _{L^2} \Vert \nabla \tau \Vert _{L^2} \leq C \Vert \nabla \tau \Vert _{L^2} ^2 + C \Vert \nabla ^2 u(t) \Vert _{L^2}  \Vert \nabla \tau \Vert _{L^2} \leq C (1+t) ^{- \left(2 + \alpha \right)}, \quad t>0. 
\end{equation*}    
As     
 \begin{displaymath}
     \frac{1}{2} \frac{d}{dt} \Vert \varepsilon (t) \Vert _{L^2} ^2 +   \Vert \varepsilon (t) \Vert _{L^2} ^2  \leq \frac{1}{2} \frac{d}{dt} \Vert \varepsilon (t) \Vert _{L^2} ^2 +   \Vert \varepsilon (t) \Vert _{L^2} ^2 + (1 - \omega) \Vert \nabla \varepsilon (t) \Vert _{L^2} ^2, 
 \end{displaymath}
we obtain that

\begin{displaymath}
    \frac{1}{2} \frac{d}{dt} \Vert \varepsilon (t) \Vert _{L^2} ^2 +   \Vert \varepsilon (t) \Vert _{L^2} ^2 \leq C (1+t) ^{- \left(2 + \alpha \right)}, \quad t>0, 
\end{displaymath}
which after integration leads to our result.

\end{proof}

\section{Proof of Theorem \ref{thm:lower}} 
\begin{proof}
Inserting 
\begin{displaymath}
\mbox{div} \, \tau = \mbox{div} \, \varepsilon + 2 \omega \, \Delta u, 
\end{displaymath}
in the equation for $u$ in \eqref{eqn:oldroyd-b} and applying the Helmholtz projector leads to  the  equation
\begin{displaymath}
    \partial_t u+ \mathbb{P} \left((u \ \cdot \nabla)u \right) = (1+\omega) \Delta u + \mathbb{P} \left( div \, \varepsilon \right).
\end{displaymath}
The regularity of our solutions allows us to write

\begin{displaymath}
u(t) = e^{(1 + \omega)t \Delta} u_0 - \int _0 ^ t e^{(1 + \omega)(t - s) \Delta} \mathbb{P} \left((u \ \cdot \nabla)u \right) (s) \, ds \, - \int _0 ^ t e^{(1 + \omega)(t-s) \Delta} \mathbb{P} \left( div \, \varepsilon \right) \, ds.
\end{displaymath}
After applying $\nabla$ and passing the derivative onto the kernel in the integral terms, we obtain

\begin{align*}
    \left\Vert \nabla u (t) -  \nabla \left( e^{(1 + \omega)t \Delta} u_0 \right) \right\Vert _{L^2} ^2 & \leq C \int_0 ^ t \Vert \nabla^ 2 e^{(1 + \omega)(t-s) \Delta} \mathbb{P} (u \otimes u) (s) \Vert _{L^2} ^2 \, ds \\ & + C \int_0 ^ t \Vert \nabla^ 2 e^{(1 + \omega)(t-s) \Delta} \mathbb{P} \varepsilon (s) \Vert _{L^2} ^2ds, \quad t>0. 
\end{align*}
We then see that 

\begin{align}
\label{eqn:estimate-nonlinear-part-du}
    \int_0 ^ t \Vert \nabla^ 2 e^{(1 + \omega)(t-s) \Delta} \mathbb{P} \,  (u \otimes u) (s) \Vert _{L^2} ^2 \, ds & \leq  C \int_0 ^ t \Vert \nabla^ 2 e^{(1 + \omega)(t-s) \Delta} \Vert _{L^ 2} ^ 2 \Vert  \,  (u \otimes u)\, (s) \Vert _{L^1} ^2 \, ds \notag \\ & \leq  C \int_0 ^ t \Vert \nabla^ 2 e^{(1 + \omega)(t-s) \Delta} \Vert _{L^ 2} ^ 2 \Vert u (s) \Vert _{L^2} ^4 \, ds \notag \\ & \leq  C \int _0 ^ t (t - s) ^ {- \frac{7}{2}} (1+s) ^{- 2 \left(\min \left\{   \frac{3}{2} + \min \{  r^{\ast} \left(u_0 \right),  1 +  r^{\ast} \left(\tau_0 \right) \}, \frac{3}{2}  \right\} \right)}  \, ds \notag \\ & \leq C  \int _0 ^ t (t - s) ^ {- \frac{7}{2}} (1+s) ^{- 2 \left( \frac{3}{2} +   r^{\ast} \left(u_0 \right) \right)}  \, ds \leq C \,  t^ {- \frac{5}{2}}, \quad t>0,
\end{align}
where we used the  assumption $r^{\ast} (u_0) \leq 1 + r^{\ast} (\tau _0)$, $r^{\ast} (u_0) \leq 0$ as well as the heat kernel estimate

\begin{displaymath}
    \Vert \nabla^ k e^ {t \Delta} f \Vert _{L^ q} \leq C t^ {- \frac{1}{2} \left(k +  n \left(\frac{1}{p} - \frac{1}{q} \right) \right)} \Vert f \Vert _{L^ p}, \quad t>0, \quad 1 < p \leq q \leq \infty. 
\end{displaymath}
Given $t > 1$, \eqref{eqn:estimate-nonlinear-part-du}  and  Theorem \ref{thm:error} imply  
\begin{displaymath}
    \Vert \nabla u (t) \Vert _{L^ 2} ^ 2   \geq \Vert \nabla e^ {t \Delta} u_0 \Vert _{L^ 2} ^ 2 - \Vert \nabla u (t) - \nabla e^ {t \Delta} u_0 \Vert _{L^ 2} ^ 2 \geq C (1 + t) ^ {- \left( \frac{5}{2} + r^{\ast} \right)} - C (1 + t) ^ {- \frac{5}{2}}  \geq C (1 + t) ^ {- \left( \frac{5}{2} + r^{\ast} \right)}, \notag
\end{displaymath}
and also

\begin{displaymath}
    \Vert \tau (t) \Vert _{L^ 2} ^ 2  \geq \Vert 2 \omega \, \nabla u(t) \Vert _{L^ 2} ^ 2 - \Vert \tau (t) - 2 \omega \, \nabla u(t) \Vert _{L^ 2} ^ 2 \geq  C (1 + t) ^ {- \left( \frac{5}{2} + r^{\ast} \right)} - C (1 + t) ^ {- \frac{7}{2}} \geq C (1 + t) ^ {- \left( \frac{5}{2} + r^{\ast} \right)}, \notag
\end{displaymath}
The proof of estimate \eqref{eqn:lower-bounds-for-final-estimate-second-case}, which corresponds to $1 + r^{\ast} (\tau _0) \leq r^{\ast} (u_0)$ and $1 + r^{\ast} (\tau _0) \leq 0$,  follows along the same lines and is hence omitted. 
\end{proof}

\textbf{Acknowledgments}
M. Hieber would like to thank the German Science Foundation DFG for support through the Research Unit FOR5528. T.H. Nguyen acknowledges support by DFG (Research Unit FOR5528) during his visit in TU Darmstadt.   C.J.  Niche acknowledges support from CNPq, Projeto Universal 408751/2023-1. C. F.  Perusato was partially supported by Cnpq Projeto Universal 421573/2023-6, Propesq - UFPE, 05/2023 (Qualis A) and Cnpq Bolsa PQ, Brazil 441618/2023-5. He is also grateful for the warm hospitality and the inspiring academic atmosphere at Technische Universität Darmstadt, where part of this work was further developed and completed.

\end{document}